\documentclass[12pt]{article}

\usepackage{graphicx,amssymb,amstext, amsthm}
\usepackage{stackrel}
\usepackage{amsmath,amsfonts}
\usepackage[round]{natbib}
\usepackage{caption}
\usepackage{float}
\usepackage{subcaption}
\usepackage{multirow}

\usepackage[english]{babel}
\usepackage[latin1]{inputenc}
\usepackage{times}
\usepackage{verbatim}
\usepackage{epstopdf,epsfig}

\newtheorem{theorem}{Theorem}
\newtheorem{lemma}{Lemma}
\newtheorem{remark}{Remark}
\newtheorem{corollary}{Corollary}
\newtheorem{Example}{Example}
\newtheorem{Definition}{Definition}

\begin{document}

\title{On Locally Dyadic Stationary Processes}
\author{Theodoros Moysiadis\\
    Institute of Applied Biosciences, Centre for Research and Technology Hellas\\
    and \\
    Konstantinos Fokianos \\
    Department of Mathematics \& Statistics, University of Cyprus}
\date{ Submitted: April 2015 \\  First Revision: April 2016 \\ Second Revision: September 2016}

\maketitle

\begin{abstract}
\noindent
We introduce the concept of local dyadic stationarity, to account for non-stationary time series, within the framework of Walsh-Fourier analysis. We define and study the time varying dyadic ARMA models (tvDARMA). It is proven that the general tvDARMA process can be approximated locally by either a tvDMA and a tvDAR process.
\end{abstract}

\noindent%
{\it Keywords:}  dyadic stationarity, local stationarity, spectral density, stationarity, Walsh functions, Walsh-Fourier analysis.
\vfill

\newpage

\section{Introduction}
\label{S:Intro}

The concept of stationarity is crucial in the statistical theory of time series analysis, especially for the development of asymptotic theory. However, the assumption of stationarity is often not realistic in applications. For example, a time series can display significant changes through time and therefore stationarity is a questionable assumption. One of the most important consequences, is that attempts to develop asymptotic results are, generally speaking, groundless, since future information of the process does not necessarily contain any information regarding the present of the process. In addition, there is no natural generalization of stationarity to non-stationarity, since non-stationary processes might exhibit trend or/and periodicity and other types of non-standard behaviour.

\citet{priestley1965evolutionary}
considered non-stationary processes whose characteristics are changing slowly over time and developed the theory of evolutionary spectra
(see \citet{priestley1981spectral,priestley1988non}).
However, such an approach makes it difficult to obtain asymptotic results, which are needed for developing estimation theory.
In order to apply standard asymptotic theory for non-stationary processes, Dahlhaus, in a series of contributions, introduced an appropriate theoretical framework, based on the concept of local stationarity (see, for example,
\citet{dahlhaus1996kullback,dahlhaus1997fitting,dahlhaus2000}).
The definition of local stationarity is based on the existence of a time varying spectral representation (\citet{dahlhaus1996kullback}).
\citet{dahlhaus2012locally}
gives an excellent and detailed overview of the theory of locally stationary processes.
A comparison between the methodology developed by Dahlhaus and Priestley is discussed in
\citet{dahlhaus1996asymptotic}.
Some other works related to locally stationary time series include the works by
\citet{granger1964spectral},
\citet{tjostheim1976spectral},
\citet{martin1981line},
\citet{melard1989contributions},
\citet{neumann1997wavelet},
\citet{nason2000wavelet},
\citet{ombao2002slex},
\citet{sakiyama2004discriminant}
and
\citet{davis2006structural},
among others.

The main goal of this contribution is to utilize the idea of local stationarity, in the sense of the above mentioned papers, for studying the spectral behaviour of time series, based on the system of Walsh functions.
These functions led to the development of Walsh-Fourier (square wave) analysis, just like the sinusoidal functions led to Fourier (trigonometric) analysis.
The motivation behind Walsh-Fourier analysis was the need to approximate stationary time series, which display square waveforms with abrupt switches (e.g. in communications and engineering), see \citet{Stankovicetal(2005)} for instance. We introduce the concept of local stationarity but based on the orthogonal system of Walsh functions to account for such phenomena that exhibit, in addition, non-stationary behaviour. We study important general classes of time series, similar in concept with the time varying ARMA (tvARMA) process- see \citet{dahlhaus2012locally}. We anticipate that our theory and methods will be applicable to non-stationary data observed in diverse applications like
pattern recognition for binary images, linear system theory and other (see  \citet{Stankovicetal(2005)}, for more).

The Walsh functions were introduced by
\citet{walsh1923closed}.
They take only two values, $+1$ and $-1$, and have similar properties with the trigonometric functions (although they are not periodic). The introduction of Walsh functions has been followed by a series of papers, related to their mathematical properties and generalizations
(\citet{fine1949walsh}),
which provided the theoretical framework for various applications, see e.g.
\citet{beauchamp1984walsh},
\citet{stoffer1991walsh}
and
\citet{abbasi2012fpga}, among others.
\citet{stoffer1991walsh} gives an excellent account of the history of Walsh functions and a comparison between Walsh and Fourier analysis.

The statistical analysis of stationary time series via Walsh functions has been based on real and dyadic time.
The dyadic time is based on the concept of dyadic addition (see Subsection \ref{S:Preli:SS:DAD}). For time points $m,n$, the real time sum $m+n$ is now replaced by the dyadic sum $m\oplus n.$
\citet{morettin1981walsh}
reviewed work on Walsh spectral analysis in both time scales. Walsh-Fourier analysis of real time stationary processes has been studied by
\citet{kohnI1980spectral,kohnII1980spectral},
\citet{morettin1983note}
and
\citet{stoffer1985central,stoffer1987walsh,stoffer1990multivariate},
among others.
The dyadic time stationarity is defined in respect to the real time stationarity as in Subsection
\ref{S:Preli:SS:DS}
(see also \citet{nagai1977dyadic}).
Further references related to the Walsh-Fourier analysis of dyadic stationary processes are
\citet{morettin1974walsh,morettin1978estimation,morettin1981walsh},
\citet{nagai1980finite},
\citet{nagai1987walsh} and
\citet{taniguchi1989statistical}.
In particular,
\citet{morettin1974walsh,morettin1978estimation}
studied the finite Walsh transform, considered the Walsh periodogram as an estimator of the Walsh spectrum and studied its theoretical properties.
\citet{nagai1977dyadic} proved that a dyadic stationary process has always unique spectral representation in terms of the system of Walsh functions and studied the dyadic linear process (see also \citet{morettin1974walsh}).
\citet{nagai1980finite} also studied dyadic autoregressive and moving average processes and their relation.

In this article, we introduce the concept of local dyadic stationarity and discuss the advantages and the perspectives of such consideration in the framework of Walsh functions.
In Section $2$ of this article, we recall some definitions and review some fundamental results for dyadic stationary processes. In Section $3$, we introduce the concept of local dyadic stationarity and study the time varying dyadic moving average process. In Section $4$, we define the general class of time varying dyadic autoregressive moving average processes and show that they exhibit locally dyadic stationarity. The article concludes with several remarks concerning further research in this topic.

\section{Preliminaries}
\label{S:Preli}

\subsection{Dyadic addition}
\label{S:Preli:SS:DAD}
We recall the definition of dyadic addition and of a dyadic process following \citet{kohnI1980spectral}.
Consider $m$ and $n$ to be non-negative integers that have the following dyadic expansions
\begin{equation}
\label{S:Preli:SS:DAD:E:dyadic expansions}
\nonumber    m=\sum_{k=0}^{f} m_k 2^k, \; n=\sum_{k=0}^{f} n_k 2^k, \quad  \textrm{where} \; m_k, n_k \in \{0,1\}.
\end{equation}
Then, the dyadic sum $m \oplus n$ is defined as
\begin{equation}
\label{S:Preli:SS:DAD:E:dyadic addition}
\nonumber    m \oplus n= \sum_{k=0}^{f} |m_k - n_k| 2^k.
\end{equation}
Consider now $x$ and $y$ to be real numbers that belong to the interval $I\equiv[0,1).$
We write
\begin{equation}
\label{S:Preli:SS:DAD:E:inverse dyadic expansions}
\nonumber    x=\sum_{k=1}^{\infty} x_k 2^{-k}, \; y=\sum_{k=1}^{\infty} y_k 2^{-k}, \quad \textrm{where} \; x_k, y_k \in \{0,1\}.
\end{equation}
In general, each of the above representations is not unique. We follow the convention that if, e.g. $x,$ can be written both through a finite or an infinite order representation, then choose the representation where $x_k=0,\;\forall \, k>k_0$ .
With this convention, the dyadic sum $x \oplus y$ is defined as
\begin{equation}
\label{S:Preli:SS:DAD:E:dyadic sum}
\nonumber    x \oplus y= \sum_{k=1}^{\infty} |x_k - y_k| 2^{-k}.
\end{equation}
Recall that the $k$-$th$ Rademacher function is $\phi_k(x)=(-1)^{x_{k+1}}, \, \forall x \in I, \, \forall \, k \geq 0.$
Then the system of Walsh functions, ${W(n,x),\; n=0,1,2,\ldots,\; x\in I},$ is defined as follows.
If $n=0,$ set $W(0,x)=1, \; \forall x\in I.$
For $n>0$, let  $n=2^{n_{1}}+2^{n_{2}}+\ldots+2^{n_{\nu}},$ where
$0\leq n_1 < n_2 \ldots < n_\nu$ . Then
\begin{equation}
\nonumber W(n,x) = \left \{\begin{array}{cc}
                               1, & n=0,  \\
                               \phi_{n_{1}}(x)\phi_{n_{2}}(x)\cdots \phi_{n_{\nu}}(x), & n>0,  \\
                               \end{array} \right.   \quad \forall x \in I.
\end{equation}
We mention briefly some characteristic properties of the Walsh functions.   \\
(i) The system of Walsh functions is orthonormal in $I$, that is
  \begin{equation}
  \label{S:Preli:SS:DAD:E:orthonormality of WF}
  \nonumber \int_0^1 W(n,x)W(m,x)dx=\left \{\begin{array}{cc}
                               1 & \textrm{for} \quad n=m,  \\
                               0 & \textrm{for} \quad n\neq m, \\
                               \end{array}\right.
  \end{equation}
and constitutes a complete set. If $f(x), \, x \in I$ is a square integrable function, then it can be expanded in a Walsh-Fourier series, i.e.
$$
f(x)=\sum_{n=0}^{\infty}c_n W(n,x),$$
with $c_n= \int_0^1 f(x)W(n,x)dx.$  \\
(ii) $\forall n,m \in \mathbb{N}, x,y \in I, \; W(n,x)W(m,x)=W(n\oplus m,x)\;$ and $\;W(n,x)W(n,y)=W(n,x\oplus y).$

The above properties of Walsh functions motivate the study of stationary time series in terms of this basis. It is well known that a second order stationary process $\{X_t, \, t\in\mathbb{N}\}$ is represented as

\begin{equation}
\label{S:Preli:SS:DAD:E:spectral representation in the Frequency Domain}
\nonumber
X_t=\int_{-\pi}^{\pi} e^{i\lambda t} dZ(\lambda),
\end{equation}
with $\{Z(\lambda), \; \lambda \in (-\infty,\infty) \}$ an orthogonal-increment process such that

\begin{equation}
\label{S:Preli:SS:DAD:E:Orthogonal-increment process}
\nonumber
E\big(dZ(\lambda) d \overline{Z(\mu)}\big)=
\eta(\lambda-\mu) dF(\lambda) d\mu
\end{equation}
\noindent
where $\eta(\cdot)$ is the Dirac function periodically extended to $\mathbb{R}$ with period $2 \pi$ and $F(\cdot)$ is the spectral distribution function (see \citet{brillinger(1974)}, for instance).	

The Walsh functions can be used instead to represent $X_t$ under the concept of dyadic stationarity. There are differences though between real and dyadic stationary processes; see \citet{morettin1981walsh} for further discussion. The concept of dyadic stationarity is explained briefly next.

\subsection{Dyadic stationarity}
\label{S:Preli:SS:DS}

We call a stochastic process $\{X_t, \, t\in\mathbb{N}\}$ dyadic stationary if it has constant mean, finite second moment
and its covariance function
\begin{equation}
\label{S:Preli:SS:DS:E:covariance function}
\nonumber R(n,m)=\textrm{cov}(X_n,X_m)=E\big[(X_n-E[X_n])(X_m-E[X_m])\big],\quad n,m \in \mathbb{N},
\end{equation}
is invariant under dyadic addition, i.e. it depends only on $n\oplus m.$
Hence, we  write for notational convenience $R(\tau)=R(n,n\oplus \tau).$
In the following assume that $E[X_t]=0,\; E[X_t^2]=1,\; \forall t \in \mathbb{N}.$
We recall some important results about dyadic stationary processes.

A dyadic stationary process $X_t$ has a dyadic spectral representation given by
(\citet[p. 193]{morettin1974walsh})
\begin{equation}
\label{S:Preli:SS:DS:E:Nagai spectral representation}
\nonumber    X_t=\int_0^1 W(t,x)dZ_X(x),\quad t\in\mathbb{N},
\end{equation}
where $\{Z_X(x), \; x\in I\}$ is a real random process with orthogonal increments, such that
$$E[dZ_X(x_1) dZ_X(x_2)]=\eta(x_1 \oplus x_2) dG_X(x_1)dx_{2},$$
where $\eta(\cdot)$ is the Dirac function periodically extended to $\mathbb{R}$ with period 1.
The function $G_X(\cdot)$, defined on $I$, is a unique distribution function, which is called the dyadic spectral distribution  of the process $X_t.$ In addition
\begin{equation}
\label{S:Preli:SS:DS:E:Nagai covariance representation}
\nonumber    R(\tau)=\int_0^1 W(\tau,x)dG_X(x).
\end{equation}
If $G_X(\cdot)$ is absolutely continuous, then $dG_X(x)=g_X(x) dx,$ where $g_X(x)$ is called the dyadic spectral density of $X_t.$

\begin{Example} \rm
\label{S:Preli:SS:DS:Ex:white noise}
A simple example of a dyadic stationary process is a sequence $\{\varepsilon_t, \, t\in\mathbb{N}\}$ of independent random variables with $E(\varepsilon_t)=0 \; \textrm{and} \; E(\varepsilon_t^2)=\sigma^2 , \; \forall \;t\in\mathbb{N}$. It is straightforward to show that its covariance function is
\begin{equation}
\label{S:Preli:SS:DS:E:covariance of white noise}
\nonumber    E\left(\varepsilon_n \varepsilon_{n\oplus\tau} \right)= R(\tau)
=\left\{ \begin{array}{ll}
\sigma^2, & \textrm{if}\;\tau=0,    \\
0, & \textrm{if}\;\tau\neq0.
\end{array} \right.
\end{equation}
\end{Example}
\noindent
Since the sequence $\varepsilon_t$ is dyadic stationary, it has a dyadic spectral representation of the form
\begin{equation}
\label{S:Preli:SS:DS:E:spectral representation of iid}
\nonumber    \varepsilon_t=\int_0^1 W(t,x) dZ_\varepsilon(x),\quad t\in\mathbb{N},
\end{equation}
with
$$E[(dZ_\varepsilon(x))^2] = dG_\varepsilon(x) = \sigma^2 dx,\quad x\in I.$$
This example illustrates the analogy of dyadic and real time stationary processes. Indeed, it is well known that a white noise real time process possesses a flat spectrum; the same is true under dyadic stationarity.

A stochastic process $\{X_t, \, t\in\mathbb{N}\}$ is a linear dyadic process if it can be represented as
(\citet{morettin1974walsh})
\begin{equation}
\label{S:Preli:SS:DS:E:infinite DMA}
    X_t=\sum_{k=0}^{\infty}a_k \varepsilon_{t\oplus k},
\end{equation}
where $\varepsilon_t$ is the sequence of i.i.d. variables, as in Example \ref{S:Preli:SS:DS:Ex:white noise}, and $\{a_k, \, k\in\mathbb{N}\}$ are real numbers, which satisfy $\sum_{k=0}^{\infty}a_k^2<\infty.$
This definition is similar in spirit to the definition of the general linear process
\citet[p.415]{priestley1981spectral}.

It can be shown that a linear dyadic process of the form
\eqref{S:Preli:SS:DS:E:infinite DMA}
is dyadic stationary, because
\begin{equation}
\label{S:Preli:SS:DS:E:covariance of linear dyadic process}
    R(\tau)=\int_0^1 W(\tau,x) \left(\sigma \sum_{k=0}^{\infty}a_k W(k,x)\right)^2  dx.
\end{equation}
In addition, it has an absolutely continuous dyadic spectral distribution function and its dyadic spectral density function has the form
\begin{equation}
\label{S:Preli:SS:DS:E:spectral density function of infinite DMA}
g(x)=\sigma^2 \left(\sum_{k=0}^{\infty}a_k W(k,x)\right)^2 = \sigma^2 A^2(x),
\end{equation}
where $A(x)=\sum_{k=0}^{\infty}a_k W(k,x).$ In this case, note that $G(x)= \sigma^{2} \int_{0}^{x} A^{2}(y) dy$, for $x \in I$.
Again, we note the analogy between real time and dyadic stationarity; the above formula is identical to the formula obtained in the real time linear process model
(\citet{priestley1981spectral}).
Furthermore, if $a_q\neq0$ and $a_k=0,\;  \forall \, k>q$ in
\eqref{S:Preli:SS:DS:E:infinite DMA},
then $X_t$ is said to be a dyadic moving average process of order $q$, abbreviated as DMA$(q)$. In general, the process $X_t$ defined by \eqref{S:Preli:SS:DS:E:infinite DMA} is called a DMA$(\infty)$ process.

\section{Local Dyadic Stationarity}
\label{S:LDS}
Consider now
\eqref{S:Preli:SS:DS:E:spectral density function of infinite DMA},
and suppose, for example, that the function $A(\cdot)$ depends upon time, i.e. it has the form $A_{t,T}(\cdot),$ where $T$ denotes the sample size. $X_t$ is now reexpressed as a triangular array 
$X_{t,T}.$
We rescale $A_{t,T}(\cdot)$ from the axis of the first $T$ non negative integers ($t=1,2,\ldots,T$) to the unit interval $I$.
The reason for this rescaling will be clear later on.
The rescaled form of $A_{t,T}(\cdot)$ is denoted by $A\left( t/T,\cdot \right)$.
We give a general definition regarding local dyadic stationarity for a process $X_{t,T}$, in the spirit of Dahlhaus (e.g \citet{dahlhaus1996kullback,dahlhaus1997fitting}).

\begin{Definition}
\label{S:LDS:D:LDS}
A sequence of stochastic processes $\{X_{t,T}, \;t=1,2,\ldots,T\}$ is called locally dyadic stationary with transfer function $A_{t,T}(\cdot)$ and trend function $\mu(\cdot),$ where $A_{t,T}(\cdot)$ and  $\mu(\cdot)$ are deterministic functions, if there exists a representation
\begin{equation}
\label{S:LDS:E:Definition, Spectral representation}
 X_{t,T}=\mu \left(\frac{t}{T}\right)+\int_0^1  W(t,x) A_{t,T}(x)  d U(x),
\end{equation}
where the following hold:\\
(i) $U(x)$ is a real-valued stochastic process on $I$ and
\begin{equation}
\label{S:LDS:E:cumulants}
\nonumber
\textrm{cum}\{dU(x_1),\ldots,dU(x_k)\} = \eta(x_1 \oplus x_2 \oplus\ldots \oplus x_k) g_k(x_1,\ldots,x_{k-1}) dx_1,\ldots,dx_k,
\end{equation}
where $\textrm{cum}\{\ldots\}$ denotes the $k$-${th}$ order cumulant, $g_1 \equiv 0, \; g_2(x_1) \equiv 1, \; |g_k(x_1,\ldots,x_{k-1})|$ are bounded for all $k$ and $\eta(\cdot)$ denotes the
Dirac delta function periodically extended to $\mathbb{R}$ with period 1.
\\
(ii) There exists a constant $K$ and a function $A:[0,1]\times\mathbb{R}\rightarrow \mathbb{R}$ such that
\begin{equation}
\label{S:LDS:E:Definition, assumption (ii)}
\underset{t,x}{\sup} \left|A_{t,T}(x)-A \left(\frac{t}{T},x\right)\right| \leq \frac{K}{T},\quad \forall \;T.
\end{equation}
The functions $A(u,x)$ and $\mu(u)$ are assumed to be continuous with respect to $u=t/T$.
\end{Definition}

The above definition is analogous  to the definition given by
\cite{dahlhaus1996kullback}.
The first condition states that the $U(x)$ has moments of order $k$; the functions $g_k(\cdot)$ are the $(k-1)$-$th$ polyspectrum of $U(x)$
following
\citet{brillinger1965introduction}.
The second assumption requires that the transfer function $A_{t,T}(\cdot)$ is approximated locally by a function $A(t/T,\cdot),$
which is the transfer function of a dyadic stationary process.
Note that the continuity of $A(u,x)$ and $\mu(u)$ in $u$ is required for the process $X_{t,T}$ to exhibit locally dyadic stationary behaviour.
Furthermore, without loss of generality, we assume that $g_{2}(x) \equiv 1$ because  the transfer function can be always rescaled  such that the
process $\{ U(x), x \in I\}$ is white noise. Indeed, the boundedness assumption of $g_{2}(.)$ implies again \eqref{S:LDS:E:Definition, assumption (ii)} for the
rescaled transfer function.

\begin{Example} \rm
Suppose $Y_t$ is a dyadic stationary process with dyadic spectral representation
$$Y_t = \int_0^1  W(t,x) A(x)  d Z(x),$$
where $E|Z^k(x)| < \infty, \, \forall \, k >0$
Define $X_{t,T}$  by
$$X_{t,T} = \mu\left(\frac{t}{T}\right) + \sigma\left(\frac{t}{T}\right)Y_t,$$
where $\mu(\cdot),\sigma(\cdot)$ are continuous functions defined on $I\rightarrow \mathbb{R}$.
Then
$$X_{t,T} = \mu\left(\frac{t}{T}\right) + \int_0^1  W(t,x) A_{t,T}(x)  d Z(x),$$
where $A_{t,T}(x)=A(t/T,x)=\sigma(t/T)A(x)$
and the assumptions $(i)$ and $(ii)$ are satisfied.
Hence $X_{t,T}$ is locally a dyadic stationary process.
\end{Example}

Consider now the process $\{X_{t,T}, \, t=1,2,\ldots,T\}$
\begin{equation}
\label{S:LDS:E:rescaled infinite tvDMA}
X_{t,T}=\sum_{k=0}^{\infty}a_{k,t,T} \varepsilon_{t\oplus k}.
\end{equation}
where $\varepsilon_t$ is an i.i.d. sequence and $\{a_{k,t,T}, \, k\in\mathbb{N}\}$ is a time-dependent process of real numbers such that $\forall t, \; \sum_{k=0}^{\infty}a_{k,t,T}^2<\infty$.
We call this process a time varying dyadic moving average process of infinite order (tvDMA$(\infty)$).
If we set in
\eqref{S:LDS:E:rescaled infinite tvDMA}
$a_{q,t,T}\neq0$ and $a_{k,t,T}=0,\, \forall \, k>q, \, \forall t,$ then we call $X_{t,T}$ a time varying dyadic moving average process of order $q$ (tvDMA$(q)$).
We rescale now the parameter curves $a_{k,t,T}$ to the unit interval $I$, assuming that there exist functions $a_{k}(t/T):I\rightarrow \mathbb{R}$ that satisfy
$a_{k,t,T} \approx a_{k}(t/T).$
We further assume that $a_{k}(\cdot)$ satisfy some regularity conditions (see Remark \ref{S:LDS:R:difference of the a functions}).
The reasons for the rescaling are described in detail, e.g. in
\citet[Sec.2]{dahlhaus2012locally}.
Briefly, suppose that we choose $a_{k,t,T}$ to be polynomials of $t$. Then, as $t\rightarrow\infty$, $a_{k,t,T}\rightarrow\infty$ as well, which violates the condition $\sum_{k=0}^{\infty}a_{k,t,T}^2<\infty.$ In addition, rescaling enables us to impose smoothing conditions through the continuity of the functions $a_{k}(\cdot),$ ensuring that the process exhibits locally dyadic stationary behaviour. Indeed,
the number of observations within the neighbourhood of a fixed point $u_0 \in I$ increases as $T\rightarrow\infty$ enabling to develop and apply locally for $X_{t,T}$ asymptotic results for dyadic stationary processes.
Suppose that the process $X_{t,T}$ defined by \eqref{S:LDS:E:rescaled infinite tvDMA} is written as
\begin{equation}
\label{S:LDS:E:rescaled infinite tvDMA_2nd}
\nonumber    X_{t,T}=\sum_{k=0}^{\infty}a_k\left(\frac{t}{T}\right) \varepsilon_{t\oplus k}.
\end{equation}
We assume that $a_k(u)=a_k(0)$ for $u<0$ and $a_k(u)=a_k(1)$ for $u>1$ and that the functions $a_{k}(\cdot)$ satisfy some  standard smoothness conditions; see \citet{dahlhaus1997fitting}.
Consider now a fixed point $u_0=t_0/T$ and its neighborhood $[u_0\pm\epsilon].$ If the length of this segment is sufficiently  small, the process 
$X_{t,T}$ can be approximated by the process $\tilde{X}_t(u_0),$ which is defined as
\begin{equation}
\label{S:LDS:E:DS infinite tvDMA}
\nonumber    \tilde{X}_t(u_0)=\sum_{k=0}^{\infty}a_k(u_0) \varepsilon_{t\oplus k},
\end{equation}
where $a_k(u_0)$ are constants, with $u_0$ indicating their dependence from the fixed point $u_0$ (see also
\citet{dahlhaus2012locally}).
$\tilde{X}_t(u_0)$ is dyadic stationary. Indeed, we can write
\begin{eqnarray}
\label{S:LDS:E:spectral representation of infinite tvDMA at a fixed point}
 \tilde{X}_t(u_0)&=& \int_0^1  W(t,x)\left( \sum_{k=0}^{\infty}a_k(u_0) W(k,x)\right) d Z_\varepsilon(x)=\int_0^1  W(t,x) A(u_0,x)
 d Z_\varepsilon(x),
\end{eqnarray}
where $A(u_0,x)=\sum_{k=0}^{\infty}a_k(u_0) W(k,x).$

From equations \eqref{S:Preli:SS:DS:E:covariance of linear dyadic process}
and \eqref{S:Preli:SS:DS:E:spectral density function of infinite DMA},
$\tilde{X}_t(u_0)$
has covariance function
\begin{equation}
\nonumber \label{S:LDS:E:local covariance of infinite tvDMA}
    R(u_0, \tau)=\int_0^1 W(\tau,x) \left(\sigma \sum_{k=0}^{\infty}a_k(u_0) W(k,x)\right)^2 dx,
\end{equation}
and a unique dyadic spectral density function given by
\begin{equation}
\label{S:LDS:E:local spectral density function of rescaled infinite tvDMA}
\nonumber
g\left(u_0,x\right)= \left(\sigma \sum_{k=0}^{\infty}a_k \left(u_0\right) W(k,x)\right)^2=\sigma^2 A^2\left(u_0,x\right),
\end{equation}
where $u_0$ indicates the dependence from a fixed point.

We can show that for
$\{ u=t/T:|t/T-u_0|\leq \epsilon \},$
it holds
$|X_{t,T}-\tilde{X}_t(u_0)|= O_P(1/T)$,
see Corollary \ref{S:LDS:C:LDS for tvDMA infty}.
Therefore, we can say that $X_{t,T}$ has locally the same covariance and dyadic spectral density function as $\tilde{X}_t(u_0)$ and therefore exhibits locally dyadic stationary behaviour.
Note that the tvDMA$(\infty)$ process $X_{t,T}$ in
\eqref{S:LDS:E:rescaled infinite tvDMA}
is locally dyadic stationary due to Definition
\ref{S:LDS:D:LDS},
since it has a time varying spectral representation as in
\eqref{S:LDS:E:Definition, Spectral representation}. Indeed, we have
\begin{eqnarray}
\label{S:LDS:E:spectral representation of infinite tvDMA}
\nonumber
    X_{t,T}&=&\sum_{k=0}^{\infty}  \left(a_{k,t,T} \int_0^1 W(t\oplus k,x)d Z_\varepsilon(x)  \right)
             =\int_0^1  W(t,x) A_{t,T}(x)  d Z_\varepsilon(x)  ,
\end{eqnarray}
where $A_{t,T}(x)=\sum_{k=0}^{\infty}a_{k,t,T}W(k,x)$ is the time varying transfer function.
We show in Theorem \ref{S:LDS:T:approximation by DS process} that, in general, a locally dyadic stationary process is approximated by a dyadic stationary process within a given interval.

\begin{theorem}
\label{S:LDS:T:approximation by DS process}
Suppose that $\{X_{t,T}, \;t=1,2,\ldots,T\}$ is a sequence of stochastic processes  which satisfy a representation of the form
\eqref{S:LDS:E:Definition, Spectral representation}
where $A_{t,T}(x)$ is the time varying transfer function (set $\mu \left(t/T\right)=0$).
Suppose that $\{\tilde{X}_{t}(u_0), \;t=1,2,\ldots,T\}$ is a dyadic stationary process with

\begin{equation}
\label{S:LDS:E:spectral representation of infinite tvDMA at a fixed point_Theorem}
\tilde{X}_{t}(u_0)=\int_0^1  W(t,x) A(u_0,x)  d U(x),
\end{equation}
where $A(u_0,x)$
depends on the fixed point $u_0 \in I$.
Then within an interval $(u_0\pm \epsilon)$ and under the assumptions of Definition
\ref{S:LDS:D:LDS}
it holds that
\begin{equation}
\label{S:LDS:E:approximation by DS process}
\nonumber    |X_{t,T}-\tilde{X}_t(u_0)|= O_P(1/T).
\end{equation}
\end{theorem}

\begin{proof}
From equations
\eqref{S:LDS:E:Definition, Spectral representation}
and
\eqref{S:LDS:E:spectral representation of infinite tvDMA at a fixed point_Theorem}
we have that

\begin{eqnarray}
\label{S:LDS:E:approximation by DS process 2}
\nonumber |X_{t,T}-\tilde{X}_t(u_0)| &=& \left| \int_0^1  W(t,x) A_{t,T}(x) d U(x) - \int_0^1  W(t,x) A(u_0,x)  d U(x) \right|       \\
\nonumber &\leq&  \int_0^1  \left| W(t,x) \right| \cdot \left| A_{t,T}(x)-  A(u_0,x) \right| d U(x)       \\
          &=&   \int_0^1  \left| A_{t,T}(x)-  A(u_0,x) \right| d U(x),
\end{eqnarray}
since $W(t,x) \in \{-1,1\}.$
In addition
\begin{eqnarray}
\label{S:LDS:E:transfer functions inequality 1}
\nonumber \left| A_{t,T}(x)-  A(u_0,x) \right| &\leq& \left| A_{t,T}(x)-  A\left( \frac{t}{T},x \right) \right| +
\left| A\left( \frac{t}{T},x \right) - A(u_0,x) \right|      \\
 &\leq&   \frac{K}{T}  + \left| A\left( u,x \right) - A(u_0,x) \right|,
\end{eqnarray}
from \eqref{S:LDS:E:Definition, assumption (ii)} in assumption (ii) of Definition \ref{S:LDS:D:LDS}. However, the same assumption states that $A(u,x)$ is continuous. Therefore, since $\{ u=t/T:|t/T-u_0|\leq \epsilon \}$ and for any $\epsilon^\prime>0$ we can choose $\epsilon>0$ to be such that $|A\left( u,x \right) - A(u_0,x)|<\epsilon^\prime,$
\eqref{S:LDS:E:transfer functions inequality 1} becomes
\begin{equation}
\label{S:LDS:E:transfer functions inequality 2}
\left| A_{t,T}(x)-  A(u_0,x) \right| \leq \frac{K^\ast}{T},
\end{equation}
for some positive constant $K^\ast.$
Finally, from \eqref{S:LDS:E:approximation by DS process 2} and \eqref{S:LDS:E:transfer functions inequality 2}, we obtain that
\begin{equation}
\label{S:LDS:E:approximation by DS process with the limit}
\nonumber E|X_{t,T}-\tilde{X}_t(u_0)| = O(1/T),
\end{equation}
and hence we have the desired result.
\end{proof}

\begin{figure}
\begin{minipage}[h]{1\linewidth}
\centering
  \begin{tabular}{@{}cc@{}}
    \includegraphics[width=.5\textwidth]{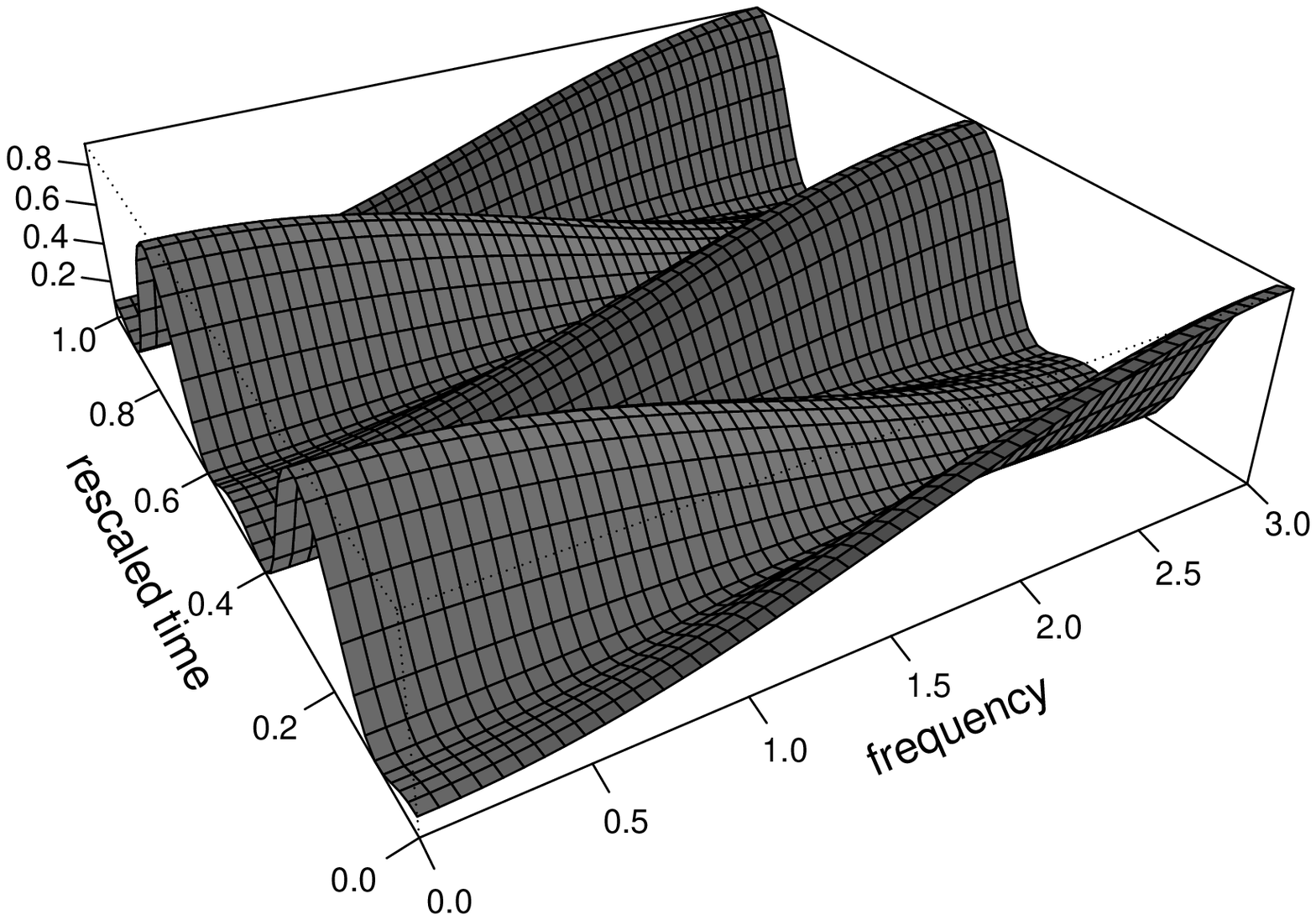} &
    \includegraphics[width=.5\textwidth]{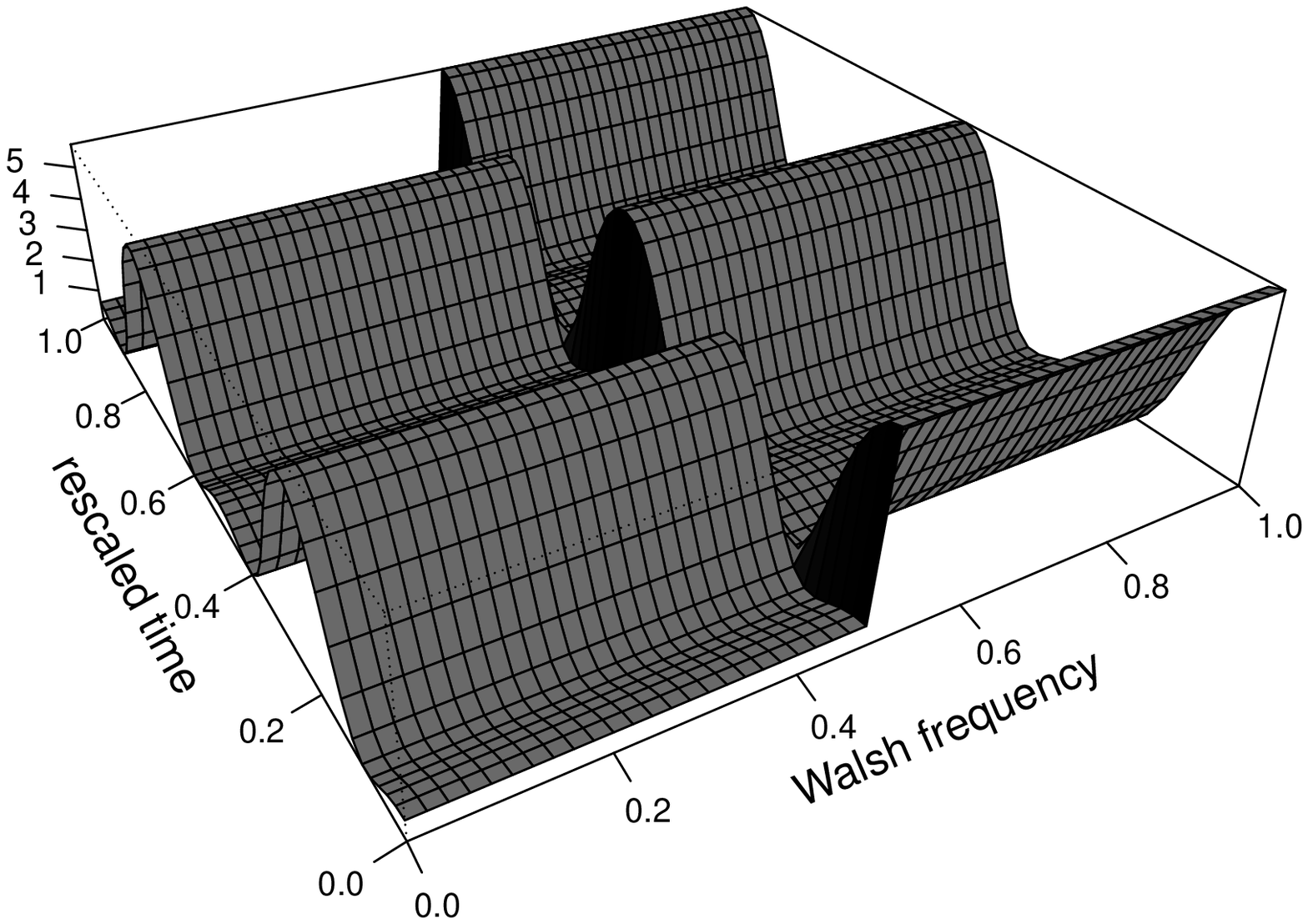}
    \end{tabular}
  \caption{The time varying spectral density function for the tvMA($1$) process (left) and the time varying dyadic spectral density function for the tvDMA($1$) process (right).}
\label{S:LDS:F:tvMA(1) Vs tvDMA(1)}
\end{minipage}
\end{figure}

\begin{corollary}
\label{S:LDS:C:LDS for tvDMA infty}	
Theorem
\ref{S:LDS:T:approximation by DS process}
holds for the tvDMA$(\infty)$ process $X_{t,T},$ defined by
\eqref{S:LDS:E:rescaled infinite tvDMA},
since this process satisfies the assumptions of Theorem
\ref{S:LDS:T:approximation by DS process},
for $\tilde{X}_t(u_0)$ given by
\eqref{S:LDS:E:spectral representation of infinite tvDMA at a fixed point}.

\end{corollary}

\begin{remark}
Theorem \ref{S:LDS:T:approximation by DS process} implies that a locally dyadic stationary process could be approximated by dyadic stationary processes within different intervals in $I$ (that may overlap). Thus its behaviour could be described via the behaviour of those dyadic stationary processes.
\end{remark}

\begin{remark}
\label{S:LDS:R:difference of the a functions}
Equation
\eqref{S:LDS:E:Definition, assumption (ii)} implies a similar assumption for the $\underset{t,x}{\sup}\left|a_{k,t,T}-a_k \left(t/T\right)\right|$ and the above discussion still holds.
\end{remark}

\begin{Example} \rm
Consider, for example, the infinite time varying MA (tvMA$(\infty)$) representation
$X_{t,T}=\sum_{k=0}^{\infty}a_{k,t,T} \varepsilon_{t- k},$ in the real time.
Then its time varying spectral density function is given by
$$f\left(u,\lambda\right)=(\sigma^2/2\pi) \left(\sum_{k=0}^{\infty}a_k(u)\exp(-i\lambda k)\right)^2.$$
Respectively, the time varying dyadic spectral density function of the tvDMA$(\infty)$ is given by
$$g\left(u,x\right)=
\left(\sigma \sum_{k=0}^{\infty}a_k(u) W(k,x)\right)^2.$$
We compare the behaviour of functions $g(u,x)$ and $f(u,\lambda)$ for the same order of the respective processes and for the same representation of the time varying coefficients $a_k(u)$ (set $\sigma^2=1$).
Figure \ref{S:LDS:F:tvMA(1) Vs tvDMA(1)} shows the spectral density function of a tvMA$(1)$ and  tvDMA$(1)$ processes. We set
$a_0(u)=-1.8\cos(1.5-\cos(4\pi u))$ and $a_1(u)=0.81.$
Figure \ref{S:LDS:F:tvMA(2) Vs tvDMA(2)} shows the spectral density function of a tvMA$(2)$ and tvDMA$(2)$ processes. In this case we set
$a_0(u)=1.2\cos(2\pi u), \; a_1(u)=2\cos(1.5-\cos(8\pi u))$ and $a_2(u)=u.$
Both figures reveal the differences between real and dyadic stationarity. The square waveform of Walsh functions allows a more oscillatory behaviour of the dyadic spectral density function.
\end{Example}

%

\begin{figure}
\begin{minipage}[h]{1\linewidth}
\centering
  \begin{tabular}{@{}cc@{}}
    \includegraphics[width=.5\textwidth]{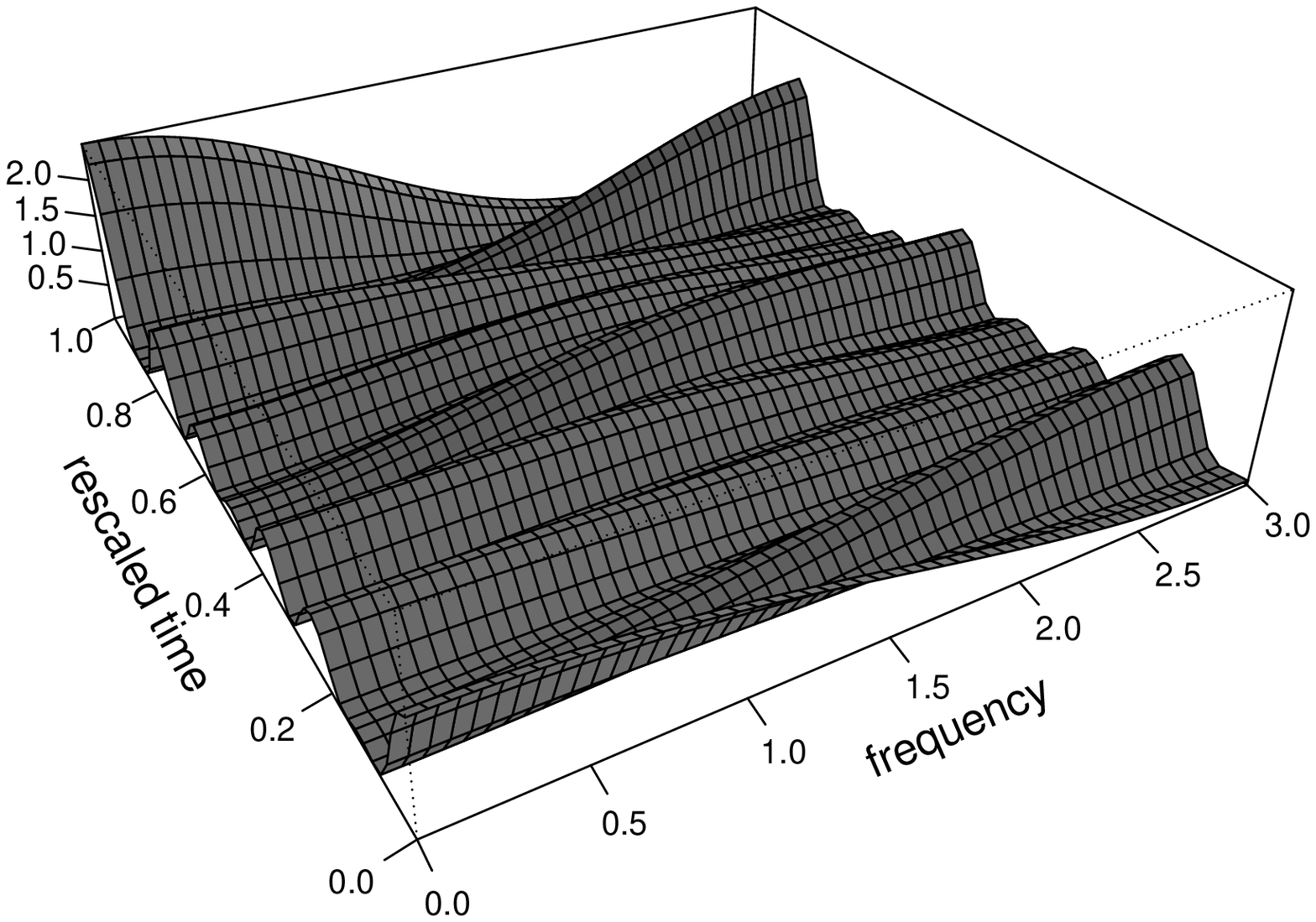} &
    \includegraphics[width=.5\textwidth]{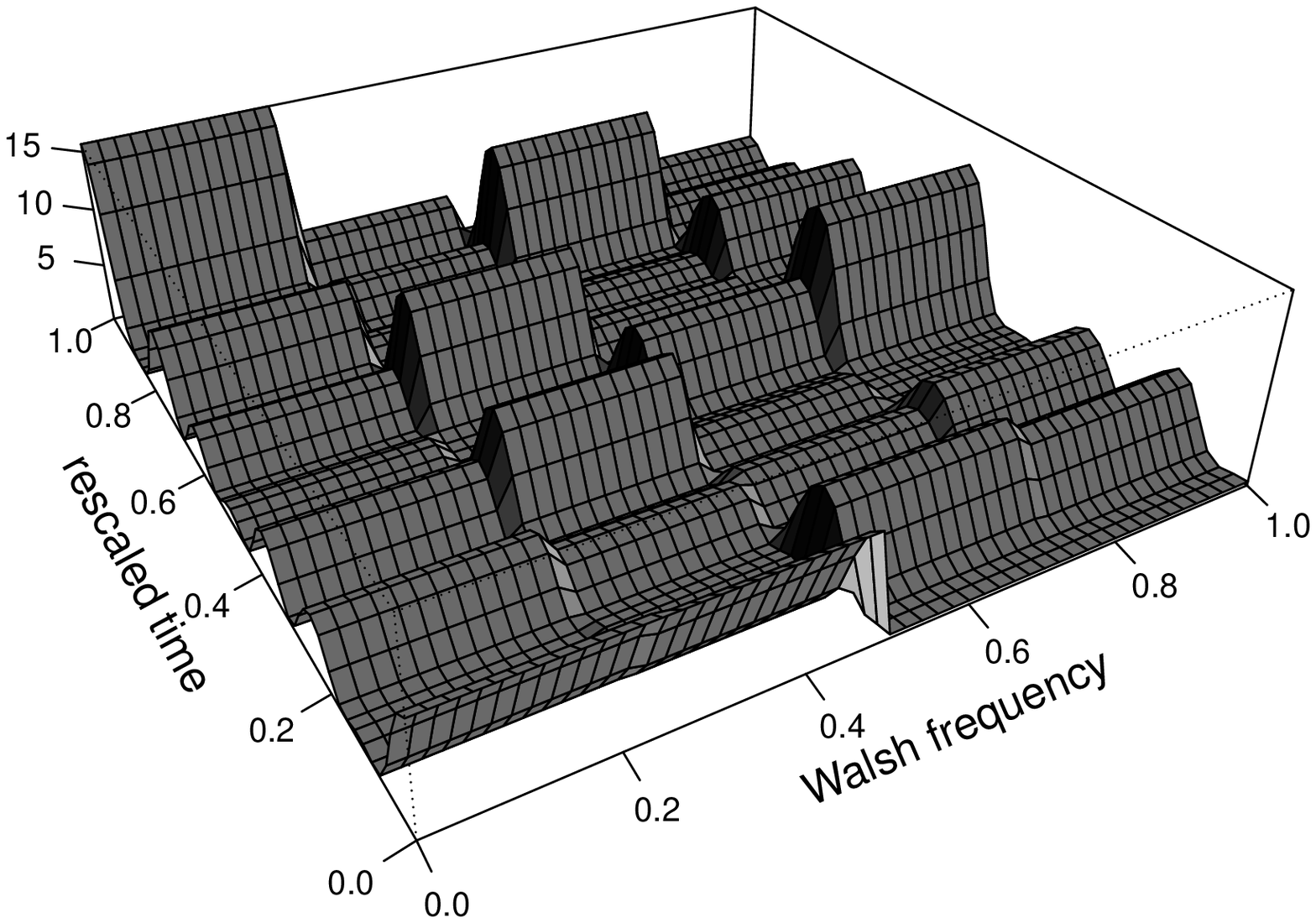}
    \end{tabular}
  \caption{The time varying spectral density function for the tvMA($2$) process (left) and the time varying dyadic spectral density function for the tvDMA($2$) process (right).}
\label{S:LDS:F:tvMA(2) Vs tvDMA(2)}
\end{minipage}
\end{figure}

\section{tvDARMA processes}
\label{S:tvDARMA}

It is well known that autoregressive, moving average, and ARMA models can be regarded as special cases of the general linear process.
\citet{nagai1980finite}
shows that a dyadic autoregressive process of finite order is always inverted into a dyadic moving average process of
finite order, and vice versa.
We obtain similar results, but within a time varying framework.
We define the time varying, dyadic, autoregressive, moving average (tvDARMA) process as follows.
\begin{Definition}
A stochastic process $\{X_t,t=1,2,\ldots,T\}$ is called tvDARMA$(p,r)$ if it is locally dyadic stationary and can be represented by

\begin{equation}
\label{S:tvDARMA:E:tvDARMA rescaled}
\sum_{k=0}^p b_{k,t,T}X_{t\oplus k,T} = \sum_{n=0}^r a_{n,t,T} \, \varepsilon_{t\oplus n},
\end{equation}
where $p,r \in \mathbb{Z}^{+}$ with $p=2^m-1,\;r=2^f-1,$ the sequences of parameters $\{b_{k,t,T}\}_{k=0,1,\ldots,p},$ $\{a_{n,t,T}\}_{n=0,1,\ldots,p}$ are real numbers with at least two non-zero parameters $b_{k_0,t,T},a_{n_0,t,T} \,$ for $2^{m-1}\leq k_0 \leq 2^m-1$ and $2^{f-1}\leq n_0 \leq 2^f-1.$
In addition, $\{\varepsilon_t,t=1,2,\ldots,T\}$
is an i.i.d. sequence with $E(\varepsilon_t)=0$ and $E(\varepsilon_t^2)=\sigma^2$.
\end{Definition}

Assume that $b_{0,t,T}=a_{0,t,T}=1.$
If we set in
\eqref{S:tvDARMA:E:tvDARMA rescaled}, $p=0$,
then the tvDMA process arises as in
\eqref{S:LDS:E:rescaled infinite tvDMA},
but for a finite order $r$. In case we set in \eqref{S:tvDARMA:E:tvDARMA rescaled} $r=0$, then \eqref{S:tvDARMA:E:tvDARMA rescaled} becomes
\begin{equation}
\label{S:tvDARMA:E:tvDARp rescaled}
\sum_{k=0}^p b_{k,t,T} X_{t\oplus k, T}=\varepsilon_t.
\end{equation}
We call $X_{t,T}$ in \eqref{S:tvDARMA:E:tvDARp rescaled} a time varying, dyadic, autoregressive process of order $p$ (tvDAR$(p)$). 
We show that a tvDAR process, and even more generally, a tvDARMA process, can be approximated by a tvDMA process. 
Following 
\cite{nagai1987walsh}, who study multivariate dyadic stationary processes, set 
\begin{equation}
\label{S:tvDARMA:E:Phi}
    \phi_{t,T}(x)=\sum_{j=0}^p b_{j,t,T} W(j,x),\;x\in I,
\end{equation}
where $p=2^m-1, m\in \mathbb{N}$ and $\{b_{j,t,T}\}_{j=0,1,\ldots,p}$ are real numbers.
Denote by $\Sigma_{t,T}$ the $(p+1)\times (p+1)$ matrix, which is given by

\begin{equation}
\label{S:tvDARMA:M:Sigma matrix}
\nonumber
    \Sigma_{t,T}=\left(
      \begin{array}{cccc}
        b_{0\oplus 0,t,T} & b_{0\oplus 1,t,T} & \cdots & b_{0\oplus p,t,T} \\
        b_{1\oplus 0,t,T} & b_{1\oplus 1,t,T} & \cdots & b_{1\oplus p,t,T} \\
        \vdots & \vdots & \ddots & \vdots \\
        b_{p\oplus 0,t,T} & b_{p\oplus 1,t,T} & \cdots & b_{p\oplus p,t,T} \\
      \end{array}
    \right).
\end{equation}

\begin{lemma}
\label{S:tvDARMA:L:Determinant}
The following equation holds
\begin{equation}
\label{S:tvDARMA:E:Determinant}
\nonumber    \det[\Sigma_{t,T}]=\prod_{j=0}^p \phi_{t,T}(x_j),
\end{equation}
where $x_j = j/(p+1), \; j=0,1,\ldots,p.$
Therefore the function $\phi_{t,T}(x)\neq 0$ if and only if $\det[\Sigma_{t,T}]\neq 0.$
\end{lemma}


\begin{lemma}
\label{S:tvDARMA:L:Invertible}
Assume that $\phi_{t,T}(x) \neq 0$ in \eqref{S:tvDARMA:E:Phi}. Then there exists a function $\eta_{t,T}(x)$, which is defined by
$$\eta_{t,T}(x)=\sum_{m=0}^p d_{m,t,T} W(m,x),\quad
\{d_{m,t,T}\}_{m=0,1,\ldots,p} \in \mathbb{R},\; x\in I,$$
and satisfies
$\phi_{t,T}(x)\eta_{t,T}(x)=1.$
The coefficients $d_{m,t,T}$ are uniquely determined by
$$\Sigma_{t,T}\left(
             \begin{array}{cccc}
               d_{0,t,T} & d_{1,t,T} & \cdots & d_{p,t,T} \\
             \end{array}
          \right)^\prime         =
          \left(
             \begin{array}{cccc}
            1 & 0 & \cdots & 0 \\
             \end{array}
          \right)^\prime.$$
\end{lemma}

Define $S_{t,T}$ to be the $(r+1)\times (r+1)$ matrix

\begin{equation}
\label{S:tvDARMA:M:S matrix}
\nonumber
    S_{t,T}=\left(
      \begin{array}{cccc}
        a_{0\oplus 0,t,T} & a_{0\oplus 1,t,T} & \cdots & a_{0\oplus r,t,T} \\
        a_{1\oplus 0,t,T} & a_{1\oplus 1,t,T} & \cdots & a_{1\oplus r,t,T} \\
        \vdots & \vdots & \ddots & \vdots \\
        a_{r\oplus 0,t,T} & a_{r\oplus 1,t,T} & \cdots & a_{r\oplus r,t,T} \\
      \end{array}
    \right).
\end{equation}

The following theorem states that a tvDARMA process can be approximated locally by a tvDMA and a tvDAR process.

\begin{theorem}
\label{S:tvDARMA:T:tvDARMA representation}
Suppose that $\{X_{t,T},t=1,2,\ldots,T\}$ is a tvDARMA$(p,r)$ as in \eqref{S:tvDARMA:E:tvDARMA rescaled}. Set $\mu=\max(p,r)$.
Then the following hold \\
(i) If $\det[\Sigma_{t,T}]\neq 0,$ then $X_{t,T}$ can be approximated locally by a tvDMA$(\mu)$ process.
\\
(ii) If $\det[S_{t,T}]\neq 0,$ then $X_{t,T}$ can be approximated locally by a tvDAR$(\mu)$ process.
\end{theorem}



\begin{corollary}
\label{S:tvDARMA:C:tvDAR represented as tv DMA}
Suppose that $\{X_{t,T},t=1,2,\ldots,T\}$ is a tvDAR$(p)$ as in \eqref{S:tvDARMA:E:tvDARp rescaled}. If $\det[\Sigma_{t,T}]\neq 0,$ $X_{t,T}$ can be approximated locally by a tvDMA$(p)$ process.
\end{corollary}

Proofs of these results are given in the appendix.

\section{Conclusions}
We anticipate that the above results will be useful for future work in the field of applications. In this direction, the concepts of Walsh spectrum and Walsh transform will be studied.
The Walsh spectrum for a real-valued dyadic stationary process $X_t$ is defined by
\begin{equation}
\label{S:LDS:SS:WT:E:defintion of WS}
    f(x)=\sum_{\tau=0}^{\infty} R(\tau)W(\tau,x),\quad 0\leq x<\infty,
\end{equation}
where the covariance function $R(\cdot)$ satisfies
$\sum_{\tau=0}^{\infty} |R(\tau)|<\infty$
(see e.g.
\citet{morettin1974walsh,morettin1978estimation,morettin1981walsh}).
Inverting
\eqref{S:LDS:SS:WT:E:defintion of WS},
the covariance is given by
$$R(\tau)=\int_0^1 W(\tau,x)f(x)dx.$$

The finite Walsh transform is given by
\begin{equation}
\label{S:LDS:SS:WT:E:defintion of WT}
\nonumber    d^{\:(N)}(x)=\sum_{n=0}^{N-1} X_n W(n,x), \quad x\in I.
\end{equation}

To estimate the Walsh spectrum,
\citet{morettin1981walsh}
defined the Walsh periodogram, by
$$I^{(N)}(x)=N^{-1} [d^{\:(N)}(x)]^2,$$
and showed that $I^{(N)}(x)$ is asymptotically an unbiased, but inconsistent, estimator of $f(x).$
He also considered the smooth Walsh periodogram and other classes of estimates.
Dyadic stationarity is necessary to estimate the time-varying  Walsh spectrum.
Therefore, in the case of local dyadic stationarity, it would be reasonable to divide the rescaled interval $I$ into subintervals and estimate the Walsh spectrum within each subinterval, where local dyadic stationarity is satisfied. The number of the observations within each subinterval $\{ u=t/T:|t/T-u_0|\leq \epsilon \}$ increases as $T$ tends to infinity and the above asymptotical results still hold. A similar method is applied by
\citet{dahlhaus1998segment}
for real-time stationary processes.

\citet{kohnI1980spectral,kohnII1980spectral} studied the system of Walsh functions for real time stationary processes. He defined the $j$-$th$ logical autocovariance, $\tau(j)$, and the corresponding  
Walsh-Fourier spectral density function $F(x)$. This notation replaces the previous notation used for $R(\cdot)$ and $f(\cdot)$ above. He considered the finite Walsh-Fourier transform  
and studied its asymptotic properties. A class of estimators for $F(x)$ was obtained, the average Walsh periodogram being a member of this class. The concept of local stationarity could also be applied in the real time setting and we conjecture that similar results could be obtained also in this case.

\section*{Acknowledgements}
We cordially thank Prof. I. Nikiforov and four anonymous reviewers for several constructive comments that improved considerably an earlier version of the manuscript.

\newpage

\section*{Appendix}
\label{TM:S:Appendix}

\renewcommand{\theequation}{A-\arabic{equation}}
\setcounter{equation}{0}
\renewcommand{\thelemma}{A-\arabic{lemma}}

\paragraph{Proof of Lemma \ref{S:tvDARMA:L:Determinant}}
Recall that $p=2^m-1$. The Walsh-ordered Hadamard matrix $H_W(m)$ is a $(2^m \times 2^m)$ matrix with elements of the form $W(n,x_j),\; x_j=j/(p+1),\;j,n=0,1,\ldots, p$, see also
\citet{stoffer1991walsh}.
Then the following relations hold:

\begin{eqnarray}
\label{TM:S:Appendix:E:lemma1 proof}
\nonumber
\Sigma_{t,T} H_W(m) &=& \left(
      \begin{array}{cccc}
        b_{0\oplus 0,t,T} & b_{0\oplus 1,t,T} & \cdots & b_{0\oplus p,t,T} \\
        b_{1\oplus 0,t,T} & b_{1\oplus 1,t,T} & \cdots & b_{1\oplus p,t,T} \\
        \vdots & \vdots & \ddots & \vdots \\
        b_{p\oplus 0,t,T} & b_{p\oplus 1,t,T} & \cdots & b_{p\oplus p,t,T} \\
      \end{array}
    \right)  \cdot
    \left( \begin{array}{cccc}
        W(0,x_0) & W(0,x_1) & \cdots & W(0,x_p) \\
        W(1,x_0) & W(1,x_1) & \cdots & W(1,x_p) \\
        \vdots & \vdots & \ddots & \vdots \\
        W(p,x_0) & W(p,x_1) & \cdots & W(p,x_p) \\
      \end{array}    \right)                        \\
\nonumber   &=& \left\{   \sum_{l=0}^p b_{(i-1)\oplus l,t,T} W(l,x_{j-1})   \right\}_{(i,j)}
             =  \bigg\{   \phi_{t,T}(x_{j-1}) W(i-1,x_{j-1})   \bigg\}_{(i,j)}   \\
            &=& H_W(m) \cdot \textrm{diag}[\phi_{t,T}(x_0), \phi_{t,T}(x_1), \ldots, \phi_{t,T}(x_p)].
\end{eqnarray}
But, since $\det [H_W(m)] = (p+1)^{(p+1)/2} \neq 0$, we get from \eqref{TM:S:Appendix:E:lemma1 proof} that
$$\det [\Sigma_{t,T}] = \det\big[ \textrm{diag}(\phi_{t,T}(x_0), \phi_{t,T}(x_1), \ldots, \phi_{t,T}(x_p)) \big]
    = \prod _{j=0}^p   \phi_{t,T}(x_j).$$

For the second argument of Lemma \ref{S:tvDARMA:L:Determinant}, note that
$\phi_{t,T}(x)=\sum_{j=0}^p b_{j,t,T} W(j,x),\, x\in I$
and every Walsh function $W(n,x),\, n=0,1,\ldots,p$ remains invariant for $x \in [x_j,x_{j+1}), \, x_j=j/(p+1),\;j,n=0,1,\ldots, p$ and equal to $W(n,x_j)$.
Therefore $\phi_{t,T}(x)=\phi_{t,T}(x_j), \; \forall x \in [x_j,x_{j+1}).$

\paragraph{Proof of Lemma \ref{S:tvDARMA:L:Invertible}}
\begin{eqnarray}
\label{TM:S:Appendix:E:lemma2 proof 1}
\nonumber   \phi_{t,T}(x)\eta_{t,T}(x) &=& \left(\sum_{l=0}^p b_{l,t,T} W(l,x)\right) \left(\sum_{m=0}^p d_{m,t,T} W(m,x)\right) =
            \sum_{l=0}^p \sum_{m=0}^p b_{l,t,T} \, d_{m,t,T} W(l\oplus m,x)     \\
   &=& \sum_{h=0}^p \Bigg [\sum_{j=0}^p b_{j\oplus h,t,T} \, d_{j,t,T} \Bigg] W(h,x).
\end{eqnarray}
In order for $\phi_{t,T}(x)\eta_{t,T}(x)=1$ to hold in $I$, we have from equation \eqref{TM:S:Appendix:E:lemma2 proof 1} that
$$\sum_{h=0}^p \Bigg [\sum_{j=0}^p b_{j\oplus h,t,T} \, d_{j,t,T} \Bigg] W(h,x_l)=1, \quad l=0,1,\ldots,p,$$
which is equivalently written in matrix notation as
\begin{equation}
\label{TM:S:Appendix:E:lemma2 proof 2}
H^\prime_W(m)
\left(  \begin{array}{c}
               \sum_{j=0}^p b_{j,t,T} \, d_{j,t,T} \\
               \sum_{j=0}^p b_{j\oplus 1,t,T} \, d_{j,t,T} \\
               \vdots \\
               \sum_{j=0}^p b_{j\oplus p,t,T} \, d_{j,t,T}
             \end{array}    \right)         =
\left(  \begin{array}{c}
               1 \\
               1 \\
               \vdots \\
               1
             \end{array}    \right).
\end{equation}
But $H_W(m)H^\prime_W(m)=2^m I_{2^m}.$ Hence, since from assumption we have that $\det\left[\Sigma_{t,T}\right]\neq 0$,
equation \eqref{TM:S:Appendix:E:lemma2 proof 2} gives that
\begin{eqnarray}
\nonumber    2^m I_{2^m}
        \left(  \begin{array}{c}
               \sum_{j=0}^p b_{j,t,T} \, d_{j,t,T} \\
               \sum_{j=0}^p b_{j\oplus 1,t,T} \, d_{j,t,T} \\
               \vdots \\
               \sum_{j=0}^p b_{j\oplus p,t,T} \, d_{j,t,T}
             \end{array}    \right)         &=&
H_W(m)  \left(  \begin{array}{c}
               1 \\
               1 \\
               \vdots \\
               1
             \end{array}    \right)
             = 2^m
        \left(  \begin{array}{c}
               1 \\
               0 \\
               \vdots \\
               0
             \end{array}    \right)     \Longrightarrow
\left(  \begin{array}{c}
               d_{0,t,T} \\
               d_{1,t,T} \\
               \vdots \\
               d_{p,t,T}
             \end{array}    \right)     = \Sigma_{t,T}^{-1}
             \left(  \begin{array}{c}
               1 \\
               0 \\
               \vdots \\
               0
             \end{array}    \right).
\end{eqnarray}

\paragraph{Proof of Theorem \ref{S:tvDARMA:T:tvDARMA representation}}
Since $X_{t,T}$  is locally dyadic stationary it has a Walsh spectral representation
$$X_{t,T}=\int_0^1 W(t,x) A_{t,T}(x) dU(x),$$
while $\varepsilon_t$ is dyadic stationary and  represented by
$$\varepsilon_t=\int_0^1 W(t,x) dZ_\varepsilon(x).$$

Then the LHS and RHS of equation \eqref{S:tvDARMA:E:tvDARMA rescaled} can be written as

\begin{eqnarray}
\label{TM:S:Appendix:E:LHS}
\textrm{LHS} &=& \sum_{k=0}^p b_{k,t,T} X_{t\oplus k,T} = \int_0^1 W(t,x) \left(\sum_{k=0}^p b_{k,t,T} A_{t\oplus k,T}(x)W(k,x) \right) dU(x),
\end{eqnarray}
and
\begin{eqnarray}
\label{TM:S:Appendix:E:RHS}
\nonumber
\textrm{RHS} &=& \sum_{n=0}^r a_{n,t,T} \varepsilon_{t\oplus n} = \int_0^1 W(t,x) \left(\sum_{n=0}^r a_{n,t,T}W(n,x) \right) dZ_\varepsilon(x).
\end{eqnarray}

\noindent
Set
\begin{equation}
\label{TM:S:Appendix:E:LHS.prime}
\textrm{LHS}^{\prime}=\int_0^1 W(t,x) \left(\sum_{k=0}^p b_{k,t,T} A_{t,T}(x)W(k,x) \right) dU(x).
\end{equation}

\noindent
Then
\begin{eqnarray}
\label{TM:S:Appendix:E:LHS.LHS.prime}
\nonumber |\textrm{LHS}-\textrm{LHS}^{\prime}| &=&
\int_0^1 |W(t,x)| \left(\sum_{k=0}^p |b_{k,t,T}| \cdot |A_{t\oplus k,T}(x)-A_{t,T}(x)| \cdot |W(k,x)| \right) dU(x) \\
&=& \int_0^1 \left(\sum_{k=0}^p |b_{k,t,T}| \cdot |A_{t\oplus k,T}(x)-A_{t,T}(x)| \right) dU(x).
\end{eqnarray}
\noindent
Consider the interval
$A=\{ \left| (t\oplus k/T)-t/T \right| \leq \varepsilon \}$.
Then, $\forall \varepsilon$, we can assume that $T$ is large enough, such that $t\oplus k/T \in A, \, \forall \, k=0,1,\ldots, p.$ From the continuity of the function $A(t/T,\cdot)$ and assumption \eqref{S:LDS:E:Definition, assumption (ii)} we have that
\begin{eqnarray}
\label{TM:S:Appendix:E:1st dif}
\nonumber |A_{t\oplus k,T}(x)-A_{t,T}(x)| &=&
\left|A_{t\oplus k,T}(x)-A\left(\frac{t\oplus k}{T},x\right)+A\left(\frac{t\oplus k}{T},x\right)-A_{t,T}(x)\right| 	\\
\nonumber &\leq&
\left| A_{t\oplus k,T}(x)-A\left(\frac{t\oplus k}{T},x\right) \right| +
\left| A\left(\frac{t\oplus k}{T},x\right)-A_{t,T}(x) \right|	\\
\nonumber &\leq&
\frac{K}{T}+
\left| A\left(\frac{t\oplus k}{T},x\right) - A\left(\frac{t}{T},x\right)\right| +
\left|A\left(\frac{t}{T},x\right) - A_{t,T}(x) \right|
\\
&\leq&
\frac{2K}{T}+\varepsilon^{\prime}=\varepsilon^{\prime\prime}.
\end{eqnarray}

From \eqref{TM:S:Appendix:E:LHS.LHS.prime} and \eqref{TM:S:Appendix:E:1st dif}, we have that

\begin{eqnarray}
\label{TM:S:Appendix:E:LHS-LHS.prime}
\nonumber
|\textrm{LHS}-\textrm{LHS}^{\prime}| &\leq& \varepsilon^{\prime\prime} \sum_{k=0}^p |b_{k,t,T}| \int_0^1dU(x)
\leq \varepsilon^{\prime\prime\prime},
\end{eqnarray}
since $\sum_{k=0}^p |b_{k,t,T}|<M^{\prime}$ and $\int_0^1dU(x)<M^{\prime\prime}$, with $M^{\prime}$ and $M^{\prime\prime}$ real constants.

Set
$\phi_{1,t,T}(x)=\sum_{k=0}^p b_{k,t,T}W(k,x)$
and
$\phi_{2,t,T}(x)=\sum_{n=0}^r a_{n,t,T}W(n,x).$
Assume that $\textrm{LHS}^{\prime}=\textrm{RHS}.$
Then, since the system of Walsh functions is complete and equation \eqref{S:tvDARMA:E:tvDARMA rescaled} holds for $t=1,2,\ldots,T$, we have that
$\phi_{1,t,T}(x) A_{t,T}(x) dU(x) = \phi_{2,t,T}(x) dZ_\varepsilon(x).$   \\
(i) Since $\det[\Sigma_{t,T}]\neq 0,$ from Lemma \ref{S:tvDARMA:L:Determinant} we have that  $\phi_{1,t,T}(x)\neq 0$ and from Lemma \ref{S:tvDARMA:L:Invertible} there exists a function $\eta_{1,t,T}(x)=\sum_{k=0}^p g_{k,t,T} W(k,x)$ such that
\begin{eqnarray}
\nonumber A_{t,T}(x) dU(x) &=& \eta_{1,t,T}(x) \phi_{2,t,T}(x)dZ_\varepsilon(x)  \\
&=&\Bigg\{
\begin{array}{cc}
\sum_{j=0}^\mu \Big( \sum_{l=0}^p g_{l,t,T}\:a_{l\oplus j,t,T} \Big) W(j,x) dZ_\varepsilon(x),& p\leq r, \\
\nonumber \sum_{j=0}^\mu \Big( \sum_{l=0}^r g_{l\oplus j,t,T}\:a_{l,t,T} \Big) W(j,x) dZ_\varepsilon(x),& p>r.
\end{array}
\end{eqnarray}
Hence,
$$\nonumber A_{t,T}(x)dU(x) =   \sum_{j=0}^\mu K_{j,t,T}W(j,x) dZ_\varepsilon(x),$$
where
\begin{eqnarray}
\label{S:tvDARMA:E:definition of Ks}
\nonumber K_{j,t,T} &=& \Bigg\{
\begin{array}{cc}
\sum_{s=0}^p  g_{s,t,T}\:a_{s\oplus j,t,T},& p\leq r, \\
\sum_{s=0}^r  g_{s\oplus j,t,T}\:a_{s,t,T},& p>r,
\end{array} \quad j=0,1,\ldots,p.
\end{eqnarray}
Therefore,
$$X_{t,T} = \sum_{j=0}^\mu  K_{j,t,T}  \int_0^1 W(t\oplus j,x) dZ_\varepsilon(x) = \sum_{j=0}^\mu  K_{j,t,T} \varepsilon_{t\oplus j}.$$     \\
(ii) Similarly with (i).

\paragraph{Proof of Corollary \ref{S:tvDARMA:C:tvDAR represented as tv DMA}}
Suppose that $X_{t,T}=\int_0^1 A_{t,T}(x)W(t,x) dU(x).$
The LHS of \eqref{S:tvDARMA:E:tvDARp rescaled} is
given by \eqref{TM:S:Appendix:E:LHS} and the $\textrm{LHS}^{\prime}$ by
\eqref{TM:S:Appendix:E:LHS.prime}.
From Lemma \ref{S:tvDARMA:L:Invertible}, since by assumption $\det[\Sigma_{t,T}]\neq 0,$
$\exists \; \eta_{t,T}(x)=\sum_{m=0}^p d_{m,t,T} W(m,x),\; d_{m,t,T}\in \mathbb{R},$ such that $\phi_{1,t,T}(x) \eta_{t,T}(x)=1.$
Therefore, and since the system of Walsh functions is complete, we have that
$$A_{t,T}(x)dU(x)= \eta_{t,T}(x) dZ_\varepsilon(x)= \sum_{m=0}^p d_{m,t,T} W(m,x) dZ_\varepsilon(x).$$
Hence, $X_{t,T}=\sum_{m=0}^p d_{m,t,T} \int_0^1 W(t\oplus m,x) dZ_\varepsilon(x)
\nonumber   = \sum_{m=0}^p d_{m,t,T} \varepsilon_{t\oplus m}.$

\end{document}